\documentclass[12pt]{amsart}
\usepackage{cite}
\usepackage{epic}
\usepackage{hyperref}
\usepackage{bm}

\newtheorem{theorem}{Theorem}[section]
\newtheorem{lemma}[theorem]{Lemma}
\newtheorem{corollary}[theorem]{Corollary}
\newtheorem{conjecture}[theorem]{Conjecture}
\theoremstyle{definition}
\newtheorem{Def}[theorem]{Definition}
\theoremstyle{remark}
\newtheorem{remark}[theorem]{Remark}
\numberwithin{equation}{section}

\DeclareMathOperator\asc{asc}

\DeclareMathOperator\des{des}
\DeclareMathOperator\plat{plat}

\DeclareMathOperator\dasc{dasc}
\DeclareMathOperator\sddes{sddes}
\DeclareMathOperator\fdesp{fdesp}
\DeclareMathOperator\ascpp{ascpp}
\DeclareMathOperator\mdup{mdup}

\DeclareMathOperator\fplat{fplat}
\DeclareMathOperator\sdes{sdes}
\DeclareMathOperator\mdes{mdes}
\DeclareMathOperator\uplat{uplat}

\DeclareMathOperator\JS{JS}

\DeclareMathOperator\Jsp{JSP}
\DeclareMathOperator\Orb{Orb}

\renewcommand{\S}{\mathfrak{S}}
\def\JSP{\mathcal{JSP}}

\newcommand{\m}{{\bf m}}
\newcommand{\Q}{\mathcal{Q}}

\newcommand{\Z}{\mathbb{Z}}

\usepackage{xcolor}

\begin{document}

\title[Stirling permutations and partial $\gamma$-positivity]
{Plateaux on generalized Stirling permutations and partial $\gamma$-positivity}

\author[Z. Lin]{Zhicong Lin}
\address[Zhicong Lin]{Research Center for Mathematics and Interdisciplinary Sciences, Shandong University, Qingdao 266237, P.R. China}
\email{linz@sdu.edu.cn}

\author[J. Ma]{Jun Ma}
\address[Jun Ma]{Department of mathematics, Shanghai Jiao Tong University, Shanghai 200240, P.R. China}
\email{majun904@sjtu.edu.cn}

\author[P.B. Zhang]{Philip B. Zhang} 
\address[Philip B. Zhang]{College of Mathematical Science, Tianjin Normal University, Tianjin 300387, P.R. China}
\email{zhang@tjnu.edu.cn}

\begin{abstract}
We prove that the enumerative polynomials of generalized Stirling permutations by the statistics of  plateaux, descents and ascents are partial $\gamma$-positive. Specialization of our result to the Jacobi-Stirling permutations confirms a recent partial $\gamma$-positivity conjecture due to Ma, Yeh and the second named author. Our partial $\gamma$-positivity expansion, as well as a combinatorial interpretation for the corresponding $\gamma$-coefficients, are obtained via the machine of context-free grammars and a  group action on generalized Stirling permutations. 
Besides, we also provide an alternative approach to the partial $\gamma$-positivity  from the  stability of certain multivariate polynomials.
\end{abstract}


\maketitle

\section{Introduction}

Let $\mathcal{A}$ be an alphabet whose elements are totally ordered. For a word $w=w_1\cdots w_n\in\mathcal{A}^n$, an index $i$, $0\leq i\leq n$, is an {\em ascent} (resp.~a {\em plateau}, a {\em descent}) of $w$ if $w_i<w_{i+1}$ (resp.~$w_i=w_{i+1}$, $w_i>w_{i+1}$), where we use the convention that $w_0=w_{n+1}=\hat{0}$. Here $\hat{0}$ is considered as an extra element smaller than all letters in $\mathcal{A}$. For instance, if $w=11211\in\mathbb{P}^5$, then $0,2$ are ascents, $1,4$ are plateaux and $3,5$ are descents of $w$. Let $\asc(w)$ (resp.~$\plat(w)$, $\des(w)$) be the number of ascents (resp.~plateaux, descents) of $w$. This paper is motivated by a partial $\gamma$-positivity conjecture of Ma, Ma and Yeh~\cite{mmy} concerning the study of these three statistics on the so-called Jacobi-Stirling permutations introduced in~\cite{glz}. 

Let us first give an overview of the Jacobi-Stirling permutations. The {\em classical Eulerian polynomials} $A_n(t)$ are defined by 
$$
\sum_{k\geq1}k^nt^k=\frac{A_n(t)}{(1-t)^{n+1}}. 
$$
It is well known~\cite{fsc0} that $A_n(t)$ can be interpreted as the descent polynomials over permutations:
$$
A_n(t)=\sum_{\pi\in\S_n}t^{\des(\pi)},
$$
where $\S_n$ is the set of all permutations of $[n]:=\{1,2,\ldots,n\}$. Recall that the {\em Stirling number of the second kind} $S(n,k)$ enumerates the set partitions of $[n]$ with $k$ blocks. In order to interpret the {\em second-order Eulerian polynomials} $C_n(t)$ appearing as
 $$
\sum_{k\geq0}S(n+k,k)t^k=\frac{C_n(t)}{(1-t)^{2n+1}}, 
$$
 Gessel and Stanley~\cite{gs} invented the Stirling permutations. A {\em Stirling permutation}  of order $n$ is a permutation of the multiset $\{1,1,2,2,\ldots,n,n\}$ such that for each $i\in[n]$, all entries between the two occurrences of $i$ are larger than $i$. Gessel and Stanley~\cite{gs} provided three different proofs for the interpretation 
 $$
 C_n(t)=\sum_{\pi\in\mathcal{Q}_n}t^{\des(\pi)},
 $$
 where $\mathcal{Q}_n$ denotes the set of Stirling permutations of order $n$. Interestingly,  the three statistics $\asc$, $\plat$ and $\des$ are equidistributed over $\Q_n$, as was shown by B\'ona in~\cite{bo} where the statistic $\plat$ was first considered. 
 
 The {\em Jacobi-Stirling numbers} $\JS(n,k;z)$, as a generalization of the Stirling number $S(n,k)$, were  introduced  in the study of a problem involving the spectral theory of powers of the classical second-order Jacobi differential expression (see~\cite{gewl}).    Write the {\em Jacobi-Stirling polynomial} $\JS(n+k,k;z)$ as $p_{n,0}(k)+p_{n,1}(k)z+\cdots+p_{n,n}(k)z^n$. Generalizing the above study for $S(n+k,k)$, Gessel, Lin and Zeng~\cite{glz} investigated the diagonal generating function 
 $$
 \sum_{k\geq0}p_{n,i}(k)t^k=\frac{A_{n,i}(t)}{(1-t)^{3n+1-i}}
 $$
 and showed that $A_{n,i}(t)$ is the descent polynomial over the Jacobi-Stirling permutations $\JSP_{n,i}$ defined in the same flavor as the Stirling permutations. Introduce the multiset
 $$M_n:=\{1,1,\bar{1},2,2,\bar{2},\ldots,n,n,\bar{n}\},$$
 where the elements are ordered by 
 $$
 \bar{1}<1<\bar{2}<2<\cdots<\bar{n}<n. 
 $$
 Let $[\bar{n}]:=\{\bar{1},\bar{2},\ldots,\bar{n}\}$. For any subset $S\subseteq[\bar{n}]$, we set $M_{n,S}=M_n\setminus S$. A permutation of $M_{n,S}$ is a {\em Jacobi-Stirling permutation} if for each $i\in[n]$, all entries between the two occurrences of $i$ are larger than $i$. We denote by $\JSP_{n,S}$ the set of all Jacobi-Stirling permutations of $M_{n,S}$ and set 
 $$
 \JSP_{n,i}=\bigcup_{S\subseteq[\bar{n}]\atop |S|=i}\JSP_{n,S}. 
 $$
 Note that $\JSP_{n,n}=\mathcal{Q}_n$.

Gamma-positive polynomials arise frequently in enumerative combinatorics and have recent motivation coming from geometry; see the survey of Athanasiadis~\cite{Ath}. 
A univariate  polynomial $f(x)$  is said to be {\em $\gamma$-positive} if it can be expanded as 
$$
f(x)=\sum_{k=0}^{\lfloor\frac{n}{2}\rfloor}\gamma_k x^k(1+x)^{n-2k}
$$
 with $\gamma_k\geq0$.
A bivariate polynomial  $h(x,y)$ is said to be {\em homogeneous $\gamma$-positive}, if $h(x,y)$ is homogeneous and $h(x,1)$ is {\em $\gamma$-positive}. This is equivalent to say that   $h(x,y)$ can be expressed as 
 $$
 h(x,y)=\sum_{k=0}^{\lfloor\frac{n}{2}\rfloor}\gamma_k(xy)^k(x+y)^{n-2k}
 $$
 with $\gamma_k\geq0$.
 One of the typical  examples due to Foata and Sch\"uzenberger~\cite{fsc0} is the {\em bivariate Eulerian polynomial} $A_n(x,y)=\sum_{\pi\in\S_n}x^{\asc(\pi)}y^{\des(\pi)}$. A trivariate  polynomial $p(x,y,z)=\sum_is_i(x,y)z^i$ is said to be {\em partial $\gamma$-positive} if every $s_i(x,y)$ is homogeneous $\gamma$-positive. 
Based on empirical evidence, Ma, Yeh and the second named author~\cite{mmy} proposed the following conjecture.
 
 \begin{conjecture}[Ma, Ma and Yeh~\cite{mmy}]\label{conj:mmy}
 Let
 $$
 \Jsp_{n,i}(x,y,z):=\sum_{\pi\in\JSP_{n,i}}x^{\asc(\pi)}y^{\des(\pi)}z^{\plat(\pi)}
 $$
 be a trivariate  extension of $A_{n,i}(y)$. For any $n\geq1$ and $0\leq i\leq n$, the polynomial  $\Jsp_{n,i}(x,y,z)$ is partial $\gamma$-positive. 
 \end{conjecture}
 
 This conjecture has already been confirmed by Ma et al.~\cite{mmy} in the special $i=0$ and $i=n$ cases. As a main contribution of this paper, we will prove a generalization of Conjecture~\ref{conj:mmy} for the Stirling permutations of  a fixed multiset. We need further definitions before we can state our main result. 

For each vector $\m=(m_1,m_2,\ldots,m_n)\in\mathbb{P}^n$, denote by $M_{\m}$ the general multiset $\{1^{m_1},2^{m_2},\ldots,n^{m_n}\}$, where $i$ appears $m_i$ times. A permutation of $M_{\m}$ is a {\em generalized Stirling permutation} if all entries between any two  occurrences of $i$ are larger than $i$ for each $i\in[n]$. Let $\mathcal{Q}_{\m}$ be the set of all  generalized Stirling permutations of $M_{\m}$.    Note that $\mathcal{Q}_{\m}=\mathcal{Q}_n$ when $\m=(2,2,\ldots,2)$. The generalized Stirling permutations and various statistics over them have been studied in~\cite{br,hm,jkp}. In particular, Brenti~\cite{br} showed that the descent polynomial over $\mathcal{Q}_{\m}$ has only real roots for each $\m\in\mathbb{P}^n$. We will consider more statistics on  $\Q_{\m}$. 

\begin{Def}[Statistics on $\Q_{\m}$]
Let $\pi=\pi_1\pi_2\cdots\pi_m\in\Q_{\m}$, where $m=\sum_{i=1}^{n}m_i$. As usual, we set $\pi_0=\pi_{m+1}=0$. A letter $k\in[n]$ is said to be {\em multiple} if $m_k>1$; and {\em single}, otherwise. An index $i\in[m]$ is called a {\em multiple} (resp.~{\em single}) {\em descent} of $\pi$ if $\pi_i>\pi_{i+1}$ and $\pi_i$ is multiple (resp.~single).  A {\em double-ascent} (resp.~{\em double-descent, peak, ascent-plateau, descent-plateau}) of $\pi$ is an index $i\in[m]$ such that $\pi_{i-1}<\pi_i<\pi_{i+1}$ (resp.~$\pi_{i-1}>\pi_i>\pi_{i+1}$, $\pi_{i-1}<\pi_i>\pi_{i+1}$, $\pi_{i-1}<\pi_i=\pi_{i+1}$, $\pi_{i-1}>\pi_i=\pi_{i+1}$). 
It is clear that if $i$ is a peak, then $\pi_i$ must be single. We further distinguish a double-descent $i$ to be single or multiple according to $\pi_i$ is single or multiple. A descent-plateau $i$ is {\em free} if there does not exist an integer $\ell$, $1\leq \ell<i$, such that $\pi_{\ell}=\pi_i$. A plateau $i$ of  $\pi$ is said to be {\em unmovable} if $i$ is neither a free descent-plateau nor an ascent-plateau. 
Let us introduce the statistics of $\pi$ by 
\begin{itemize}
\item $\dasc(\pi)=\#\{i\in[m]: \pi_{i-1}<\pi_i<\pi_{i+1}\}$, the number of {\bf d}ouble-{\bf asc}ents of $\pi$;
\item $\sddes(\pi):=\#\{i\in[m]: \pi_{i-1}>\pi_i>\pi_{i+1}\text{ and $\pi_i$ is single}\}$, the number of {\bf s}ingle {\bf d}ouble-{\bf des}cents of $\pi$;
\item $\fdesp(\pi):=\#\{i\in[m]: \pi_{i-1}>\pi_i=\pi_{i+1}\text{ and $\pi_{\ell}\neq\pi_i$ for all $\ell\in[i-1]$}\}$, the number of {\bf f}ree {\bf des}cent-{\bf p}lateaus of $\pi$;
\item $\ascpp(\pi):=\#\{i\in[m]: \pi_{i-1}<\pi_i=\pi_{i+1}\text{ or $\pi_{i-1}<\pi_i>\pi_{i+1}$}\}$, the number of {\bf asc}ent-{\bf p}lateaus and {\bf p}eaks. 
\item $\mdup(\pi):=\#\{i\in[m]: \text{ $i$ is a {\bf m}ultiple {\bf d}escent or a {\bf u}nmovable {\bf p}lateau}\}$.
\end{itemize}
For example, if $\pi=15565333124411$, then $\dasc(\pi)=2$, $\sddes(\pi)=0$, $\fdesp(\pi)=1$, $\ascpp(\pi)=4$ and $\mdup(\pi)=4$.
\end{Def}

Now we are ready to state our main result. 
\begin{theorem}\label{thm:stir}
For any $\m\in\mathbb{P}^n$ with $m_1+m_2+\cdots m_n=m$, let 
$$
S_{\m}(x,y,z)=\sum_{\pi\in\Q_{\m}}x^{\asc(\pi)}y^{\des(\pi)}z^{\plat(\pi)}.
$$
The polynomial $S_{\m}(x,y,z)$ is partial $\gamma$-positive and  has the expansion
$$
S_{\m}(x,y,z)=\sum_{i=0}^{m-1}z^i\sum_{j=1}^{\lfloor\frac{m+1-i}{2}\rfloor}\gamma_{\m,i,j}(xy)^j(x+y)^{m+1-i-2j},
$$
where 
\begin{equation}\label{gam:stir}
\gamma_{\m,i,j}=\#\{\pi\in\Q_{\m}: \sddes(\pi)=\fdesp(\pi)=0, \mdup(\pi)=i,\ascpp(\pi)=j\}. 
\end{equation}
\end{theorem}

\begin{remark}
To see that Theorem~\ref{thm:stir} implies Conjecture~\ref{conj:mmy}, for any $S\subseteq[n]$ with $|S|=i$, define $\m(S)=(m_1,\ldots,m_{2n-i})$ where 
$$
m_{\ell}=
\begin{cases}
2, \quad \text{if $\ell=p+|\{a\in[n]\setminus S:a\leq p\}|$ for some $1\leq p\leq n$};\\
1,\quad\text{otherwise}.
\end{cases}
$$
For instance, if $S=\{1,2,5,7\}\subseteq[7]$, then $\m(S)=(2,2,1,2,1,2,2,1,2,2)$. Let 
$$\Jsp_{n,S}(x,y,z):=\sum_{\pi\in\JSP_{n,S}}x^{\asc(\pi)}y^{\des(\pi)}z^{\plat(\pi)}.$$ It is routine to check that $\Jsp_{n,S}(x,y,z)=S_{\m}(x,y,z)$ with $\m=\m(S)$. Therefore, by Theorem~\ref{thm:stir}, $\Jsp_{n,S}(x,y,z)$ is partial $\gamma$-positive, namely, 
$$
\Jsp_{n,S}(x,y,z)=\sum_{k=0}^{3n-i-1}z^k\sum_{j=1}^{\lfloor\frac{3n-i+1-k}{2}\rfloor}\gamma_{\m,k,j}(xy)^j(x+y)^{3n-i+1-k-2j}.
$$
Conjecture~\ref{conj:mmy} then follows from $\Jsp_{n,i}(x,y,z)=\sum_{S\in[n]\atop {|S|=i}}\Jsp_{n,S}(x,y,z)$. 
\end{remark}

The rest of this paper is organized as follows. In Section~\ref{sec:main}, we provide a  proof of Theorem~\ref{thm:stir} by using the context-free grammars and a generalization of the Foata--Strehl action on Stirling permutations. 
In Section~\ref{sec:stab},
an analytic approach to the partial $\gamma$-positivity of $S_{\m}(x,y,z)$ is studied by using the stable theory of multivariate polynomials developed by Bocrea and Br\"and\'en. 
As a result, the descent polynomials over Stirling permutations with fixed number of plateaux  are shown to be real-rooted. 

\section{Proof of Theorem~\ref{thm:stir}}\label{sec:main}
\subsection{Context-free grammars and an equidistribution}
For a set $V=\{x,y,z,\ldots\}$ of commutative variables, a {context-free grammar} $G$ is a set of substitution rules that replace a variable in $V$ by a Laurent polynomial of variables in $V$. The formal derivative $D$ associated with a context-free grammar $G$ (introduced by Chen in~\cite{chen}) is defined by $D(x)=G(x)$ for any $x\in V$ and satisfies the following relations:
\begin{align*}
D(u+v)&=D(u)+D(v),\\
D(uv)&=D(u)v+uD(v),
\end{align*}
where $u$ and $v$ are two Laurent polynomials of variables in $V$.  The context-free grammars have been found useful in studying various combinatorial structures~\cite{chen,chy, cy,dum,cf,mmy,ma2}, including permutations, increasing trees, labeled rooted trees and set partitions. For example, if $V=\{x,y\}$ and $G=\{x\rightarrow xy, y\rightarrow xy\}$, then $D(x)=xy$, $D^2(x)=xy(x+y)$ and $D^3(x)=D(xy)(x+y)+D(x+y)xy=x^3y+4x^2y^2+xy^3$. This is the grammar introduced by Dumont~\cite{dum} to generate the bivariate  Eulerian polynomials $D^n(x)=A_n(x,y)$.

Let $\pi=\pi_1\pi_2\cdots\pi_m\in\Q_{\m}$, where $m=\sum_{i=1}^{n}m_i$.  A plateau $i$ of $\pi$ is called a {\em first plateau} if $\pi_j\neq\pi_i$ for all $1\leq j<i$.  Let $\fplat(\pi)$ be the number of  {\bf f}irst {\bf plat}eau of $\pi$.
Denote by $\sdes(\pi)$ (resp.~$\mdes(\pi)$) the number of {\bf s}ingle (resp.~{\bf m}ultiple) {\bf des}cents of $\pi$.
We will apply the context-free grammars to prove the following equidistribution. 

\begin{lemma}\label{lem:plat}
For any $\m\in\mathbb{P}^n$ with $m_1+m_2+\cdots+m_n=m$, the following two  triplets of  statistics are equidistributed on the Stirling permutations $\Q_{\m}$:
$$
(\des,\plat,\asc)\quad\text{and}\quad(\fplat+\sdes,\mdup,\asc).
$$
\end{lemma}
\begin{proof}
Note that $i$ is a first plateau if and only if $i$ is either a  free descent-plateau or an ascent-plateau. In other words, a plateau is unmovable if and only if it is not a first plateau. Thus, $\plat(\pi)-\fplat(\pi)=\uplat(\pi)$, the number of {\bf u}nmovable {\bf plat}eau of $\pi$. We aim to show that the quintuplets
\begin{equation}\label{equi:quin}
(\sdes,\mdes,\fplat,\uplat,\asc)\quad\text{and}\quad(\sdes,\fplat,\mdes,\uplat,\asc)
\end{equation}
are equidistributed on $\Q_{\m}$, from which the lemma follows. 

 For a Stirling permutation $\pi\in\Q_{\m}$, we first introduce a grammatical labeling of $\pi$ as follows:
\begin{itemize}
\item[($L_1$)] If $i$ is a single descent, then put a superscript label $x$ right after $\pi_i$; 
\item[($L_2$)] If $i$ is a multiple descent, then put a superscript label $\tilde{x}$ right after $\pi_i$; 
\item[($L_3$)] If $i$ is a first plateau, then put a superscript label $\tilde{y}$ right after $\pi_i$; 
\item[($L_4$)] If $i$ is a unmovable plateau, then put a superscript label $y$ right after $\pi_i$; 
\item[($L_5$)] If $i$ is an ascent, then put a superscript label $z$ right after $\pi_i$. 
\end{itemize}
Recall that we always set $\pi_0=\pi_{m+1}=0$. For example, if $\pi=15565333124411$, then the labeling of $\pi$ is 
$$
0{}^{z}1^z5^{\tilde{y}}5^z6^x5^{\tilde{x}}3^{\tilde{y}}3^y3^{\tilde{x}}1^z2^z4^{\tilde{y}}4^{\tilde{x}}1^y1^{\tilde{x}}0.
$$
It is clear that the weight $x^{\sdes(\pi)}\tilde{x}^{\mdes(\pi)}\tilde{y}^{\fplat(\pi)}y^{\uplat(\pi)} z^{\asc(\pi)}$ is the product of all the superscripts in the labelings of $\pi$. 
Let $V=\{x,\tilde{x},y,\tilde{y},z\}$. For integer $k\geq2$, introduce the context-free grammar 
$$
G_k=\{x\rightarrow\tilde{x}\tilde{y}y^{k-2}z, \tilde{x}\rightarrow\tilde{x}\tilde{y}y^{k-2}z, y\rightarrow\tilde{x}\tilde{y}y^{k-2}z, \tilde{y}\rightarrow\tilde{x}\tilde{y}y^{k-2}z,z\rightarrow\tilde{x}\tilde{y}y^{k-2}z\}.
$$
Also, define the context-free grammar 
$$
G_1=\{x\rightarrow xz, \tilde{x}\rightarrow xz, y\rightarrow xz, \tilde{y}\rightarrow xz,z\rightarrow xz\}.
$$
For any $k\geq1$, let $D_k$ be the formal derivative associated with the context-free grammar $G_k$.  
We claim that 
$$
D_{m_n}D_{m_{n-1}}\cdots D_{m_1}(z)=\sum_{\pi\in\Q_{\m}}x^{\sdes(\pi)}\tilde{x}^{\mdes(\pi)}\tilde{y}^{\fplat(\pi)}y^{\uplat(\pi)} z^{\asc(\pi)}.
$$
The equidistribution~\eqref{equi:quin} then follows from this claim and the fact that the context-free grammars $G_k$ are symmetric in $\tilde{x}$ and $\tilde{y}$ for all $k\geq1$. 

It remains to show the claim. We proceed by induction on $n$. The statement is obviously true for the initial case $n=0$, as $\Q_0=\{0^z0\}$. Suppose that we have all labeled permutations in $\Q_{\m'}$, where $\m'=(m_1,m_2,\ldots,m_{n-1})$. Note that very permutation in $\Q_{\m}$ can be constructed from a permutation $\sigma\in\Q_{\m'}$ by inserting $n^{m_n}$, $m_n$ copies of $n$, to the position between $\sigma_i$ and $\sigma_{i+1}$ for some nonnegative integer $i$. 
The changes of labelings are illustrated as follows for any $v\in\{x,\tilde{x},y,\tilde{y},z\}$ (no matter what the label $v$ is): 
\begin{itemize}
\item If $m_n\geq2$, then
$$
\cdots\sigma_i^v\sigma_{i+1}\cdots\mapsto \cdots\sigma_i^zn^{\tilde{y}}n^y\cdots n^yn^{\tilde{x}}\sigma_{i+1}\cdots,
$$
where $n^y$ appears $(m_n-2)$ times. 
\item Otherwise $m_n=1$, and 
$$
\cdots\sigma_i^v\sigma_{i+1}\cdots\mapsto \cdots\sigma_i^zn^x\sigma_{i+1}\cdots.
$$
\end{itemize}
In either case, the insertion of $n^{m_n}$ corresponds to one substitution rule in $G_{m_n}$. Since the action of $D_{m_n}$ on elements of $\Q_{\m'}$ generates all elements of $\Q_{\m}$, the claim holds. This completes the proof of the lemma. 
\end{proof}

In order to finish the proof of Theorem~\ref{thm:stir}, we need  a generalization of the so-called Foata--Strehl group action on Stirling permutations that will be introduced below. 

\subsection{Generalized Foata--Strehl actions on Stirling permutations}

The Foata--Strehl group action on permutations was invented by Foata and Strehl~\cite{fst} (see also~\cite{swg}) to prove combinatorially the homogeneous  $\gamma$-positivity expansion 
$$
A_n(x,y)=\sum_{k=1}^{\lfloor\frac{n+1}{2}\rfloor}\gamma_{n,k}(xy)^k(x+y)^{n+1-2k},
$$
where $\gamma_{n,k}$ enumerates the  permutations in $\S_n$ with $k$ descents and with no double descents. 
Since then some generalizations and analogues  of this  $\gamma$-positivity expansion, with or without combinatorial proofs, have been found~\cite{bra,flz,lk,lin,lz}. In particular, although the $\gamma$-coefficients of the {\em double Eulerian polynomials} were known~\cite{lin} to be nonnegative, to find a combinatorial interpretation of them is still widely open. Here we develop a generalization of the Foata--Strehl action for Stirling permutations to get the combinatorial interpretation of $\gamma_{\m,i,j}$ in~\eqref{gam:stir}, finalizing the proof of  Theorem~\ref{thm:stir}.

For a Stirling permutation $\pi\in\Q_{\m}$ and a value  $x\in[n]$, suppose that $\pi_{\ell}$ is the leftmost occurrence of $x$ in $\pi$. We call $x$ a {\em free descent-plateau value} (resp.~a {\em single double descents value}, a  {\em double ascents value}) of $\pi$ if  $\ell$ is a free descent plateau (resp.~a single double descents, a double ascents) of $\pi$.  
Then we introduce the {\em Generalized Foata-Strehl action} (GFS-action for short) $\varphi_x$ as follows: 
\begin{itemize}
\item  If $x$ is a free descent-plateau value or a single double descents value of $\pi$, then $\varphi_x(\pi)$ is obtained from $\pi$ by moving   $\pi_{\ell}$ to the right of the letter $\pi_k$, where $k=\max\{a: 0\leq a\leq \ell-2,\pi_a<x\}$; 
\item If $x$ is a double ascents value of $\pi$, then  $\varphi_x(\pi)$ is obtained from $\pi$ by moving  $\pi_{\ell}$ to the left of the letter $\pi_k$, where $k=\min\{a: \ell+2\leq a\leq m+1,\pi_a\leq x\}$;
\item If $x$ is not in the above two cases, then let $\varphi_x(\pi)=\pi$.
\end{itemize}
The  GFS-action has a nice visualization as depicted in Fig.~\ref{GFS-action}. For example, if $\pi=15565333124411$, then $\varphi_1(\pi)=5565333{\bf1}124411$, $\varphi_3(\pi)=1{\bf3}556533124411$ and $\varphi_{2}(\pi)=15565333144{\bf2}11$. When $m_1=m_2=\cdots=m_n=1$, the GFS-action becomes the version of {\em Modified Foata-Strehl action} introduced in~\cite{lk}.

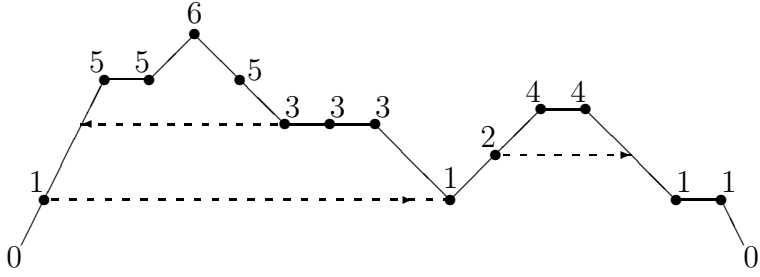
\begin{figure}
\setlength{\unitlength}{1mm}
\begin{picture}(108,38)\setlength{\unitlength}{1mm}
\thinlines
%
%
%
%

\put(9,9){\dashline{1}(1,0)(54,0)}
\put(58,9){\vector(1,0){0.1}}

\put(41,19){\dashline{1}(-1,0)(-27,0)}
\put(14,19){\vector(-1,0){0.1}}

\put(69,15){\dashline{1}(1,0)(18,0)}
\put(87,15){\vector(1,0){0.1}}

\put(6,3){\line(1,2){11}}\put(4,0){$0$}
\put(9,9){\circle*{1.3}}\put(7,10){$1$}
\put(17,25){\circle*{1.3}}\put(17,25){\line(1,0){6}}\put(15,26){$5$}
\put(23,25){\circle*{1.3}}\put(23,25){\line(1,1){6}}\put(21,26){$5$}
\put(29,31){\circle*{1.3}}\put(29,31){\line(1,-1){12}}\put(28,32.5){$6$}
\put(35,25){\circle*{1.3}}\put(36,25){$5$}
\put(41,19){\circle*{1.3}}\put(41,19){\line(1,0){12}}\put(41,20){$3$}
\put(47,19){\circle*{1.3}}\put(47,20){$3$}
\put(53,19){\circle*{1.3}}\put(53,19){\line(1,-1){10}}\put(53,20){$3$}
\put(63,9){\circle*{1.3}}\put(63,9){\line(1,1){12}}\put(62,10.5){$1$}
\put(69,15){\circle*{1.3}}\put(67,16){$2$}
\put(75,21){\circle*{1.3}}\put(75,21){\line(1,0){6}}\put(73,22){$4$}
\put(81,21){\circle*{1.3}}\put(81,21){\line(1,-1){12}}\put(79,22){$4$}
\put(93,9){\circle*{1.3}}\put(93,9){\line(1,0){6}}\put(93,10){$1$}
\put(99,9){\circle*{1.3}}\put(99,10){$1$}
\put(99,9){\line(1,-2){3}}\put(102,0){$0$}

\end{picture}
\caption{Generalized Foata--Strehl actions on $15565333124411$}
\label{GFS-action}
\end{figure}

We are now ready for the proof of Theorem~\ref{thm:stir}. 

\begin{proof}[{\bf Proof of Theorem~\ref{thm:stir}}]

By Lemma~\ref{lem:plat}, we have
$$
S_{\m}(x,y,z)=\sum_{\pi\in\Q_{\m}}x^{\asc(\pi)}y^{\fplat(\pi)+\sdes(\pi)}z^{\mdup(\pi)}.
$$
Let us define the set $\Q_{\m,i}=\{\pi\in\Q_{\m}: \mdup(\pi)=i\}$. Then Theorem~\ref{thm:stir} is equivalent to 
\begin{equation}\label{stir:gam}
\sum_{\pi\in\Q_{\m,i}}x^{\asc(\pi)}y^{\fplat(\pi)+\sdes(\pi)}=\sum_{\pi\in\widetilde{\Q_{\m,i}}}(xy)^{\ascpp(\pi)}(x+y)^{m+1-i-2\times\ascpp(\pi)},
\end{equation}
where $\widetilde{\Q_{\m,i}}:=\{\pi\in\Q_{\m,i}: \sddes(\pi)=\fdesp(\pi)=0\}$. 

Clearly, the GFS-actions $\varphi_x$'s are involutions and commute. Thus, for any $S\subseteq[n]$ we can define the function $\varphi_S:\Q_{\m}\rightarrow\Q_{\m}$ by $\varphi_S=\prod_{x\in S}\varphi_x$, where the product is the functional compositions. For instance, continuing with the example in Fig.~\ref{GFS-action}, we have $\varphi_{\{1,3\}}(\pi)={\bf3}556533{\bf1}124411$. Hence the group $\Z_2^n$ acts on $\Q_{\m}$ via the function $\varphi_S$. Since the statistic ``$\mdup$'' is invariant under this group action,  it divides $\Q_{\m,i}$ into some disjoint orbits. For each $\pi\in\Q_{\m,i}$, let $\Orb(\pi)=\{g(\pi): g\in\Z_2^n\}$ be the orbit of $\pi$ under the GFS-action. 
Note that $x$ is a free descent-plateau value or a single double descents value of $\pi$ if and only if $x$ is a double ascents value of $\varphi_x(\pi)$. Thus, there exists a unique Stirling permutation $\tilde{\pi}$ in $\Orb(\pi)$ such that $\sddes(\tilde{\pi})=\fdesp(\tilde{\pi})=0$, that is, $\widetilde{\Q_{\m,i}}\bigcap\Orb(\pi)=\{\tilde{\pi}\}$. 
Therefore, we have 
\begin{align*}
\sum_{\sigma\in\Orb(\pi)}x^{\asc(\sigma)}y^{\fplat(\sigma)+\sdes(\sigma)}&=x^{\asc(\tilde{\pi})-\dasc(\tilde{\pi})}y^{\fplat(\tilde{\pi})+\sdes(\tilde{\pi})}(x+y)^{\dasc(\tilde{\pi})}\\
&=(xy)^{\ascpp(\tilde{\pi})}(x+y)^{m+1-i-2\times\ascpp(\tilde{\pi})},
\end{align*}
where the second equality follows from the following relationships
\begin{equation}\label{eq:st1}
\asc(\tilde{\pi})-\dasc(\tilde{\pi})=\fplat(\tilde{\pi})+\sdes(\tilde{\pi})=\ascpp(\tilde{\pi})
\end{equation}
and 
\begin{equation}\label{eq:st2}
\dasc(\tilde{\pi})=m+1-i-2\times\ascpp(\tilde{\pi}).
\end{equation}
Summing over all orbits of $\Q_{\m,i}$ under the GFS-action then gives~\eqref{stir:gam}.

It remains to show the above relationships. Since $\tilde{\pi}$ has neither free descent-plateau nor single double descents, every first plateau must be an ascent-plateau and every single descent is a peak. Thus, we have $\fplat(\tilde{\pi})+\sdes(\tilde{\pi})=\ascpp(\tilde{\pi})$. As each ascent is followed immediately by an ascent, a plateau or a descent, we have $\asc(\tilde{\pi})=\dasc(\tilde{\pi})+\ascpp(\tilde{\pi})$, which proves~\eqref{eq:st1}.  Clearly, we have $\asc(\tilde{\pi})+\plat(\tilde{\pi})+\des(\tilde{\pi})=m+1$ and $\mdup(\tilde{\pi})+\fplat(\tilde{\pi})+\sdes(\tilde{\pi})=\plat(\tilde{\pi})+\des(\tilde{\pi})$. It then follows that $\mdup(\tilde{\pi})+\asc(\tilde{\pi})+\fplat(\tilde{\pi})+\sdes(\tilde{\pi})=m+1$ and so by~\eqref{eq:st1},  
\begin{equation*}
\dasc(\tilde{\pi})=\asc(\tilde{\pi})-\ascpp(\tilde{\pi})=m+1-\mdup(\tilde{\pi})-2\times\ascpp(\tilde{\pi}),
\end{equation*}
which is relationship~\eqref{eq:st2}. This completes the proof of the theorem. 
\end{proof}

\section{Stability}
 \label{sec:stab}
   In this section, we  shall  study the  analytic property of the trivariate polynomial $S_{\m}(x,y,z)$.
  From the involution
$$
\pi_1\pi_2\cdots\pi_m\mapsto\pi_m\pi_{m-1}\cdots\pi_1, 
$$
 we know that  $S_{\m}(x,y,z)$ is  symmetric in $x$ and $y$.
Hence, 
the  partial $\gamma$-positivity of $S_{\m}(x,y,z)$ is  equivalent to the  $\gamma$-positivity of the following refined descent polynomials
$$
S_{\m,i}(x)=\sum_{\pi}x^{\des(\pi)},
$$
where the sum runs over all Stirling permutations $\pi\in\Q_{\m}$ with $\plat(\pi)=i$.  
It is known (see~\cite[Remark~7.3.1]{br2}) that the real-rootedness of  a polynomial with symmetric coefficients implies the $\gamma$-positivity of such polynomial. 
 This motivates us to study an alternative approach to the partial $\gamma$-positivity of $S_{\m}(x,y,z)$.  
 
Let us  recall the  stability of multivariate polynomials, which has been developed enormously by Bocrea and Br\"and\'en~\cite{bb, bb2}.
  A polynomial $f \in \mathbb{R}[x_1, \dots, x_n]$ is said to be  \emph{stable} if either $f(x_1, \dots, x_n) \neq 0$  whenever  $ \mbox{Im}(x_i)>0$ for all $i$ or $f$ is identically zero.
  Note that a univariate real polynomial is stable if and only if it has only real roots.
Several multivariate Eulerian polynomials have been shown to be stable, see~\cite{blm, bhvw,chy, vw,zz}.
  
  Our main result of this section is stated as follows.
    \begin{theorem}\label{thm:stable}
  	Let  $p(x,y,z)$ be  a trivariate homogeneous polynomial  with nonnegative coefficients.
  	If  $p(x,y,z)$   is   stable and symmetric in $x$ and $y$, then $p(x,y,z)$ is partial $\gamma$-positive. 
  \end{theorem}

  In order  to prove Theorem~\ref{thm:stable},  we shall introduce some basic results about stability-preserving linear operators.
  \begin{lemma}[{See~\cite[Lemma 2.4]{w}}]\label{lem:stable}
  	Given $i,j \in [n]$, the following operations preserve  stability of $f \in \mathbb{R}[x_1, \dots, x_n]$:
  	\begin{enumerate}
  		\item \emph{Differentiation:} $f \mapsto \partial f/\partial x_i.$
  		\item \emph{Diagonalization:} $f \mapsto f|_{x_i=x_j}.$
  		\item	\emph{Specialization:} for $a \in \mathbb{R}$, $f\mapsto f|_{x_i = a}.$
  	\end{enumerate}
  \end{lemma}


  Now it is time for us to prove  Theorem~\ref{thm:stable}.
  \begin{proof}[{\bf Proof of Theorem~\ref{thm:stable}}]
  Suppose that 	$p(x,y,z)=\sum_{i=0}^{d}  s_i(x,y)z^i$. 
  	We first prove that for each $0\le k \le d$, the polynomial $s_k(x,y)$ is stable.
  	By taking the $k$-th order partial derivative with respect to $z$ of the real stable polynomial $p(x,y,z)$ it follows from Lemma \ref{lem:stable} that
  	$\sum_{i=k}^{d} (i)_{k}s_i(x,y) z^{i-k}$
  	is  stable, where $(i)_k=i(i-1)\cdots (i-k+1)$.
  	Note that if $z$ is in the upper half plane then so is $-1/z$. Hence, we obtain the  stability of
$\sum_{i=k}^{d} (-1)^{i-k} (i)_{k} s_i(x,y) z^{d-i}.$
  	Similarly, by taking the $(d-k)$-th  order the partial derivative  with respect to $z$, we get the  stability of $s_k(x,y) $.

We next  prove that $s_{k}(x,y)$ is homogeneous $\gamma$-positive.
By Lemma~\ref{lem:stable}, we get that $s_k(x,1)$ is real-rooted. 
Since  $s_{k}(x,y)$ is symmetric in $x$ and $y$, we know that the coefficients of $s_k(x,1)$ are symmetric.
Hence, the polynomial  $s_k(x,1)$ is  $\gamma$-positive.
Since $s_{k}(x,y)$ is homogeneous in $x$ and $y$, we obtain the homogeneous $\gamma$-positivity of $s_{k}(x,y)$. 
This completes the proof.
  \end{proof}

We proceed to use Theorem~\ref{thm:stable} to prove the partial $\gamma$-positivity of $S_{\m}(x,y,z)$. In order to do this, it suffices to show that the stability $S_{\m}(x,y,z)$. For an element $\pi\in\Q_{\m}$, let $\mathcal{D}(\pi)$ and $\mathcal{A}(\pi)$ be the set of descents and ascents of $\pi$, respectively. If $\kappa=\max_{i} m_i$ and $1\leq j<\kappa$, define $\mathcal{P}_j(\pi)$ to be the set of indices $i$ such that $\pi_i=\pi_{i+1}$ where $\pi_1\cdots\pi_{i-1}$ contains $j-1$ instances of $\pi_i$. 
Haglund and  Visontai~\cite[Theorem 3.5]{hm} showed that the multivariate polynomial
$$
\sum_{\pi\in\Q_{\m}}\prod_{i\in\mathcal{D}(\pi)}x_{\pi_i} \prod_{i\in\mathcal{A}(\pi)} y_{\pi_{i}}  \prod_{j=1}^{\kappa-1}\biggl( \prod_{i\in\mathcal{P}_j(\pi)} z_{j,\pi_i} \biggr)
$$
is stable. 
By diagonalizing the  variables $x_{\pi_i}$, $y_{\pi_{i}}$ and $z_{j,\pi_i}$ to $x$, $y$ and  $z$ respectively, 
it follows from Lemma~\ref{lem:stable} that the polynomial  
$S_{\m}(x,y,z)$
is stable.
As an application of this result we  get the real-rootedness of $S_{\m, i}(x)$.
 \begin{corollary}
 	For any $\m\in\mathbb{P}^n$ with $m_1+m_2+\cdots m_n=m$ and  any  $0\le i \le m-1$,  the polynomial $S_{\m, i}(x)$ has only real roots. 
 \end{corollary}


\section*{Acknowledgments}

The first author was supported by the National Science Foundation of China grant 11871247 and the project of Qilu Young Scholars of Shandong University.
The second author was supported by the National Science Foundation of China grant 11571235. 
The third author was supported by the National Science Foundation of China grant 11701424.


\begin{thebibliography}{99}

\bibitem{gewl} G.E. Andrews, E. Egge, W. Wolfgang and L.L. Littlejohn, The Jacobi-Stirling numbers, J. Combin. Theory Ser. A, {\bf120} (2013), 288--303.


\bibitem{Ath} C.A. Athanasiadis, Gamma-positivity in combinatorics and geometry, S\'em. Lothar. Combin. 77 (2018), Article B77i, 64pp (electronic).

\bibitem{bb}
J.~Borcea and P.~Br{\"a}nd{\'e}n, The {L}ee-{Y}ang and {P}\'olya-{S}chur
programs. {I}. {L}inear operators preserving stability, Invent. Math., {\bf177}
(2009), 541--569.

\bibitem{bb2}
J.~Borcea and P.~Br{\"a}nd{\'e}n, 
The Lee-Yang and Pólya-Schur programs. II. Theory of stable polynomials and applications,
Comm. Pure Appl. Math., 62 (2009), 1595–1631. 

\bibitem{bo} M. Bona, Real zeros and normal distribution for statistics on Stirling permutations defined by Gessel and Stanley, SIAM J. Discrete Math., {\bf23} (1) (2008/2009), 401--406.

\bibitem{blm} P. Br\"and\'en, M. Leander, M. Visontai, Multivariate Eulerian polynomials and exclusion processes. Combin. Probab. Comput., 25 (2016),  486–499. 


\bibitem{bra} P. Br\"and\'en, Actions on permutations and unimodality of descent polynomials, European J. Combin., {\bf29} (2008), 514--531. 

\bibitem{br2} P. Br\"and\'en, Unimodality, log-concavity, real-rootedness and beyond, {\em Handbook of Enumerative Combinatorics}, CRC Press Book. \href{http://arxiv.org/abs/1410.6601}{(arXiv:1410.6601)}

\bibitem{bhvw}
P. Br{\"a}nd{\'e}n, J. Haglund, M. Visontai, and D. G. Wagner, Proof of the
monotone column permanent conjecture, in Notions of positivity and the
geometry of polynomials, Trends Math., Birkh\"auser/Springer Basel AG, Basel,
2011, 63--78.

\bibitem{br} F. Brenti, Unimodal, Log-concave and P\'olya frequency sequences in combinatorics, Mem. Amer. Math. Soc., {\bf413} (1989).

\bibitem{chen} W.Y.C. Chen, Context-free grammars, differential operators and formal power series, Theoret. Comput. Sci., {\bf117} (1993), 113--129. 

\bibitem{cf} W.Y.C. Chen and A.M. Fu, Context-free grammars for permutations and increasing trees, Adv. in Appl. Math., {\bf39} (2017), 58--82. 

\bibitem{chy} W.Y.C. Chen, R.X.J. Hao, H.R.L. Yang, Context-free grammars and multivariate stable polynomials over
Stirling permutations, arXiv:1208.1420.

\bibitem{cy}W.Y.C. Chen and  H.R.L. Yang, A context-free grammar for the Ramanujan-Shor polynomials, Adv. in Appl. Math., to appear (2019).

\bibitem{dum} D. Dumont, Grammaires de William Chen et d\'erivations dans  les arbres et arborescences, S\'em. Lothar. Combin., 37 (1996), Art. B37a. 

\bibitem{fsc0}D. Foata and M.-P. Sch\"uzenberger, {\em Th\'eorie G\'eom\'etrique des Polyn\^omes Eul\'eriens},  Lecture Notes in Mathematics, Vol. 138, Springer-Verlag, Berlin, 1970.


\bibitem{fst}D. Foata and V. Strehl, Rearrangements of the symmetric group and enumerative properties of the tangent and secant numbers, Math. Z., {\bf137} (1974), 257--264.

\bibitem{flz} S. Fu, Z. Lin and J. Zeng, On two unimodal descent polynomials, Discrete Math., {\bf341} (2018), 2616--2626. 

\bibitem{glz} I. Gessel, Z. Lin and J. Zeng, Jacobi-Stirling polynomials and $P$-partitions, European J. Combin., {\bf33} (2012), 1987--2000.

\bibitem{gs} I. Gessel and R.P. Stanley, Stirling polynomials, J. Combin. Theory Ser. A, {\bf24} (1978), 24--33.

\bibitem{hm} J. Haglund, M. Visontai, Stable multivariate Eulerian polynomials and generalized Stirling permutations, European J. Combin., {\bf33} (2012), 477--487.

\bibitem{jkp} S. Janson, M. Kuba and A. Panholzer, Generalized Stirling permutations, families of increasing trees and urn models, J. Combin. Theory Ser. A, {\bf 118} (2011), 94--114.


\bibitem{lk} Z. Lin and D. Kim, A sextuple equidistribution arising in pattern avoidance, J. Combin. Theory Ser. A, {\bf 155} (2018), 267--286.

\bibitem{lin} Z. Lin, Proof of Gessel's $\gamma$-positivity conjecture, Electron. J. Combin., {\bf 23} (2016), \#P3.15.

\bibitem{lz} Z. Lin and J. Zeng, The $\gamma$-positivity of basic Eulerian polynomials via group actions, J. Combin. Theory Ser. A, {\bf 135} (2015), 112--129.

\bibitem{mmy} S.-M. Ma, J. Ma and Y.-N. Yeh, $\gamma$-positivity and partial $\gamma$-positivity of descent-type polynomials, J. Combin. Theory Ser. A, {\bf 167} (2019), 257--293.

\bibitem{ma2} S.-M. Ma, J. Ma and Y.-N. Yeh and B.-X. Zhu, Context-free grammars for several polynomials associated with Eulerian polynomials, Electron. J. Combin., {\bf 25(1)} (2018), \#P1.31.

\bibitem{swg} L. Shapiro, W.-J. Woan and S. Getu, Runs, slides and moments, SIAM J. Algebraic Discrete Methods, {\bf4} (1983), 459--466. 

\bibitem{vw}
	M. Visontai and N.~Williams, Stable multivariate {$W$}-{E}ulerian polynomials,
	J. Combin. Theory Ser. A, {\bf 120} (2013), 1929--1945.
	
\bibitem{w}
D.G. Wagner, Multivariate stable polynomials: theory and applications, Bull.
Amer. Math. Soc., {\bf48} (2011), 53--84.


\bibitem{zz} P.B. Zhang and X. Zhang, Multivariate stable Eulerian polynomials on segmented permutations, European J. Combin., {\bf78} (2019), 155-162.

\end{thebibliography}
 \end{document}